\documentclass[12pt]{amsart}
\usepackage{amscd,amsmath,amssymb,amsfonts}
\usepackage[all]{xy}
\theoremstyle{plain}
\newtheorem{thm}{Theorem}
\newtheorem{lem}[thm]{Lemma}
\newtheorem{cor}[thm]{Corollary}
\newtheorem{prop}[thm]{Proposition}

\theoremstyle{definition}
\newtheorem{defn}[thm]{Definition}

\newtheorem{question}[thm]{Question}
\newtheorem{rmk}[thm]{Remark}

\numberwithin{thm}{section} \numberwithin{equation}{section}

\newcommand{\ga}[2]{\begin{gather}\label{#1}#2 \end{gather}}

\newcommand{\surj}{\twoheadrightarrow}

\newcommand{\Hom}{{\rm Hom}}

\newcommand{\sE}{{\mathcal E}}
\newcommand{\sF}{{\mathcal F}}

\newcommand{\sL}{{\mathcal L}}

\newcommand{\sO}{{\mathcal O}}

\newcommand{\sQ}{{\mathcal Q}}

\newcommand{\sV}{{\mathcal V}}

\begin{document}

\title{Frobenius morphism and semi-stable bundles }
\author{Xiaotao Sun}
\address{Academy of Mathematics and Systems Science, Chinese Academy of Science, Beijing, P. R. of China}
\email{xsun@math.ac.cn}
\address{}
\date{March 23, 2009}
\thanks{Partially supported by NSFC (No. 10731030) and Key Laboratory of Mathematics
Mechanization (KLMM)}

\begin{abstract}This article is the expanded version of a talk given at the
conference: Algebraic geometry in East Asia 2008. In this notes, I
intend to give a brief survey of results on the behavior of
semi-stable bundles under the Frobenius pullback and direct images.
Some results are new.
\end{abstract}
\maketitle

\section{Introduction}

Let $X$ be a smooth projective variety of dimension $n$ over an
algebraically closed field $k$ with ${\rm char}(k)=p>0$. The
absolute Frobenius morphism $F_X:X\to X$ is induced by
$\mathcal{O}_X\to \mathcal{O}_X,\quad f\mapsto f^p$. Let $F:X\to
X_1:=X\times_kk$ denote the relative Frobenius morphism over $k$.
This simple endomorphism of $X$ is of fundamental importance in
algebraic geometry over characteristic $p>0$. One of the themes is
to study its action on the geometric objects on $X$. Here we
consider the pull-back $F^*$ and direct image $F_*$ of torsion free
sheaves on $X$. For example, is the semi-stability (resp. stability)
of torsion free sheaves preserved by $F^*$ and $F_*$ ? Even on
curves of genus $g\ge 2$, it is known that $F^*$ does not preserve
the semi-stability of torsion free sheaves (cf. \cite{Gi} for
example). However, it is now also know that $F_*$ preserves the
stability of torsion free sheaves on curves of genus $g\ge 2$ (cf.
\cite{Su}). In this paper, we are going to discuss the behavior of
semi-stability of torsion free sheaves under $F^*$ and $F_*$.

Recall that a torsion free sheaf $\sE$ is called semi-stable (resp.
stable) if $\mu(\sE')\le\mu(\sE)$ (resp. $\mu(\sE')<\mu(\sE)$) for
any nontrivial proper sub-sheaf $\sE'\subset\sE$, where $\mu(\sE)$
is the slope of $\sE$ (See definition in Section 3). Semi-stable
sheaves are basic constituents of torsion free sheaves in the sense
that any torsion free sheaf $\sE$ admits a unique filtration
$${\rm HN}_{\bullet}(\sE): 0={\rm HN}_0(\sE)\subset{\rm HN}_1(\sE)\subset \cdots\subset{\rm HN}_{\ell+1}(\sE)=\sE,$$
  which is the so called Harder-Narasimhan filtration, such that
  \begin{itemize}
 \item[\rm{(1)}] ${\rm gr}_i^{\rm HN}(\sE):={\rm HN}_i(\sE)/{\rm HN}_{i-1}(\sE)$ ($1\le i\le \ell+1$) are semistable;
\item[\rm{(2)}]  $\mu({\rm gr}_1^{\rm HN}(\sE))>\mu({\rm gr}_2^{\rm HN}(\sE))>\cdots>\mu({\rm gr}_{\ell+1}^{\rm HN}(\sE))$.
  \end{itemize}
The rational number ${\rm I}(\sE):=\mu({\rm gr}_1^{\rm
HN}(\sE))-\mu({\rm gr}_{\ell+1}^{\rm HN}(\sE))$, which measures how
far is a torsion free sheaf from being semi-stable, is called the
instability of $\sE$. It is clear that $\sE$ is semi-stable if and
only if ${\rm I}(\sE)=0$. Thus the main theme of this investigation
is to look for upper bound of ${\rm I}(F^*\sE)$ and ${\rm
I}(F_*\sE)$.

In Section 2, we recall the notion of connections with $p$-curvature
zero and Cartier's theorem, which simply says that a quasi-coherent
sheaf is the pullback of a sheaf if and only if it has a connection
of $p$-curvature zero. In particular, a sub-sheaf of $F^*\sE$ is the
pullback of a sub-sheaf of $\sE$ if and only if it is invariant
under the action of the canonical connection on $F^*\sE$. This is
the main tool in Section 3 to find a upper bound of $F^*\sE$.

In Section 3, we survey various upper bounds of the instability
${\rm I}(F^\sE)$ in terms of ${\rm I}(\sE)$ and numerical invariants
of $\Omega^1_X$. For curves, the bound is a linear combination of
${\rm I}(\sE)$ and $\mu(\Omega^1_X)$. For higher dimensional
varieties $X$, the difficulty to obtain such a bound lies in the
fact that tensor product of two semi-stable sheaves may not be
semi-stable in characteristic $p>0$. A theorem of A. Langer can
solve this difficulty in certain sense. He proved in \cite{L} that
there is a $k_0$ for a torsion free sheaf $\sE$ such that the
Harder-Narasimhan filtration of $F^{k*}\sE$ has strongly semi-stable
quotients whenever $k\ge k_0$. As a price of it, the upper bound is
a linear combination of ${\rm I}(\sE)$ and the limit
$$L_{\max}(\Omega^1_X)=\lim_{k\to\infty}\frac{\mu_{\max}(F^{k*}\Omega^1_X)}{p^k}.$$
It is natural to expect a upper bound in terms of ${\rm I}(\sE)$ and
$\mu_{\max}(\Omega^1_X)$ (cf. Remark \ref{rmk3.16}), but I do not
know any such bound in general.

In Section 4, we discuss the stability of $F_*W$. The main tool in
this section is the canonical filtration \eqref{4.5} of $F^*(F_*W)$,
which is again induced by the canonical connection on $F^*(F_*W)$.
After a brief proof of the main theorem in \cite{Su}, we reveal some
implications in the proof. We show that the proof itself implies
that $F_*\sL$ and the sheaf $B^1_X$ of local exact differential
$1$-forms on $X$ are stable if $\mu(\Omega^1_X)>0$ and ${\rm
T}^{\ell}(\Omega^1_X)$ ($1\le\ell\le n(p-1)$) are semi-stable. In
fact, for $\sE\subset F_*\sL$ (resp. $B'\subset B^1_X$), we show
that $\mu(\sE)-\mu(F_*\sL)$ (resp. $\mu(B')-\mu(B^1_X)$) is bounded
by an explicit negative number (cf. the inequalities \eqref{4.18}
and \eqref{4.20}). The work of M. Raynaud have revealed the
important relationship between $B^1_X$ and the fundamental group of
$X$. I do not know if the result above has any application in this
direction.

\section{Frobenius and connections of $p$-curvature zero}

Let $X$ be a smooth projective variety of dimension $n$ over an
algebraically closed field $k$ with ${\rm char}(k)=p>0$. The
absolute Frobenius morphism $F_X:X\to X$ is induced by the
homomorphism $$\mathcal{O}_X\to \mathcal{O}_X,\qquad f\mapsto f^p$$
of rings. Let $F:X\to X_1:=X\times_kk$ denote the relative Frobenius
morphism over $k$ that satisfies
$$\xymatrix{\ar@/^20pt/[rr]^{F_X} X\ar[r]^F
\ar[dr] & X_1\ar[r]\ar[d]
& X\ar[d]^{}\\
 & {\rm Spec}(k)\ar[r]^{F_k}& {\rm Spec}(k)} .$$

According to a theorem of Cartier, the fact that a quasi-coherent
$\sE$ on $X$ is the pull-back of a sheaf on $X_1$ by $F$ is
equivalent to the fact that $\sE$ has a connection of $p$-curvature
zero. Let me recall briefly the theme from \cite{K} (See Section 5
of \cite{K}).

For a quasi-coherent sheaf $\sE$ on $X$, a connection on $\sE$ is a
$k$-linear homomorphism $\nabla:\sE\to\sE\otimes_{\sO_X}\Omega^1_X$
satisfying the Leibniz rule
$$\nabla(f\cdot e)=f\nabla(e)+e\otimes df,\quad \forall\,\,f\in\sO_X,\,e\in\sE$$
where $df$ denotes the image of $f$ under $d:\sO_X\to\Omega^1_X$.
The kernel
$$\sE^{\nabla}:=ker(\nabla:\sE\to\sE\otimes_{\sO_X}\Omega^1_X)$$
is an abelian sheaf of the germs of horizontal sections of
$(\sE,\nabla)$.

Let ${\rm Der}(\sO_X)$ be the sheaf of derivations, i.e., for any
open set $U\subset X$, ${\rm Der}(\sO_X)(U)$ is the set of
derivations ${\rm D}:\sO_U\to\sO_U$. It is a sheaf of $k$-Lie
algebras and is isomorphic to ${\Hom}_{\sO_X}(\Omega^1_X,\sO_X)$ as
$\sO_X$-modules. A connection $\nabla$ on $\sE$ is equivalent to a
$\sO_X$-linear morphism
$$\nabla:{\rm Der}(\sO_X)\to {\rm End}_k(\sE)$$ satisfying
$\nabla({\rm D})(f\cdot e)={\rm D}(f)\cdot e+f\nabla({\rm D})$ where
${\rm End}_k(\sE)$ is the sheaf of $k$-linear endomorphisms of
$\sE$, which is also a sheaf of $k$-Lie algebras.

A connection $\nabla:{\rm Der}(\sO_X)\to {\rm End}_k(\sE)$ is
integrable if it is a homomorphism of Lie algebras. A morphism
between $(\sE,\nabla)$ and $(\sF,\nabla')$ is a morphism
$\Phi:\sE\to\sF$ of quasi-coherent $\sO_X$-modules satisfying
$$\Phi(\nabla({\rm D})(e))=\nabla'({\rm D})(\Phi(e)),\quad\forall
\,\,{\rm D}\in {\rm Der}(\sO_X),\,\, e\in\sE.$$ Then the pairs
$(\sE,\nabla)$ of quasi-coherent sheaves with integrable connections
form an abelian category $MIC(X)$.

Since ${\rm char}(k)=p>0$, the $p$-th iterate ${\rm D}^p$ of a
derivation ${\rm D}$ is again a derivation. Thus ${\rm Der}(\sO_X)$
and ${\rm End}_k(\sE)$ are both sheaves of restricted $p$-Lie
algebras. The $p$-curvature of an integrable connection
$$\nabla:{\rm Der}(\sO_X)\to {\rm End}_k(\sE)$$
measures how far the homomorphism $\nabla$ being a homomorphism of
restricted $p$-Lie algebras. More precisely,

\begin{defn}\label{dfn2.1} The $p$-curvature of $\nabla:{\rm Der}(\sO_X)\to {\rm End}_k(\sE)$
  is the morphism of sheaves
  $\Psi^{\nabla}:{\rm Der}(\sO_X)\to {\rm
  End}_k(\sE)$ defined by
  $$\Psi^{\nabla}({\rm D}):=(\nabla(\rm D))^p-\nabla(\rm D^p)$$
  which is in fact a morphism $\Psi^{\nabla}:{\rm Der}(\sO_X)\to {\rm
  End}_{\sO_X}(\sE)$ (i.e. $\Psi^{\nabla}({\rm D})$ is
  $\sO_X$-linear for any $\rm D \in {\rm Der}(\sO_X)$.
\end{defn}

Let $F:X\to X_1$ be the relative Frobenius morphism. Then, for any
quasi-coherent sheaf $\sF$ on $X_1$, there is a unique connection
$$\nabla_{{\rm can}}: F^*(\sF)\to F^*(\sF)\otimes_{\sO_X}\Omega^1_X,$$
which is integrable and of $p$-curvature zero,  such that
$$\sF\cong(F^*(\sF))^{\nabla_{{\rm can}}}.$$
We call $\nabla_{{\rm can}}$ the canonical connection on the
pull-back $F^*(\sF)$. It turns out that a quasi-coherent sheaf $\sE$
on $X$ having a connection of $p$-curvature zero is enough to
characterize that $\sE$ is a pull-back of a quasi-coherent sheaf on
$X_1$. More precisely, given a $(\sE,\nabla)$ of $p$-curvature zero,
the abelian sheaf $\sE^{\nabla}$ is in a natural way a
quasi-coherent sheaf on $X_1$ such that $F^*(\sE^{\nabla})\cong
\sE$. Moreover, we have

\begin{thm}\label{thm2.2}(Cartier) Let $F:X\to X_1$ be the relative Frobenius
morphism. Then the functor
$$\sF \mapsto (F^*(\sF), \nabla_{{\rm can}})$$
is an equivalence of categories between the category of
quasi-coherent sheaves on $X_1$ and the full subcategory of $MIC(X)$
consisting of $(\sE,\nabla)$ whose $p$-curvature is zero. The
inverse functor is
$$(\sE,\nabla)\mapsto \sE^{\nabla}.$$
\end{thm}

\section{Instability of Frobenius pull-back}

 Let $X$ be a smooth projective variety of dimension $n$ over an
algebraically closed field $k$ with ${\rm char}(k)=p>0$. Fix an
ample divisor ${\rm H}$ on $X$, for a torsion free sheaf $W$ on $X$,
the slope of $W$ is defined as
$$\mu(W)=\frac{c_1(W)\cdot {\rm H}^{n-1}}{{\rm rk}(W)}$$
where ${\rm rk}(W)$ denotes the rank of $W$. Then

\begin{defn}\label{defn3.1}
  A torsion free sheaf $W$ on $X$ is called semi-stable (resp. stable)
  if for any subsheaf $W'\subset W$ we have
  $$\mu(W')\,\,\le
  (resp.\,<\,)\,\,\mu(W).$$
\end{defn}

\begin{thm}\label{thm3.2}(Harder-Narasimhan filtration)
  For any torsion free sheaf $\sE$, there is a unique filtration
  $${\rm HN}_{\bullet}(\sE): 0={\rm HN}_0(\sE)\subset{\rm HN}_1(\sE)\subset \cdots\subset{\rm HN}_{\ell+1}(\sE)=\sE,$$
  which is the so called Harder-Narasimhan filtration, such that
  \begin{itemize}
 \item[\rm{(1)}] ${\rm gr}_i^{\rm HN}(\sE):={\rm HN}_i(\sE)/{\rm HN}_{i-1}(\sE)$ ($1\le i\le \ell+1$) are semistable;
\item[\rm{(2)}]  $\mu({\rm gr}_1^{\rm HN}(\sE))>\mu({\rm gr}_2^{\rm HN}(\sE))>\cdots>\mu({\rm gr}_{\ell+1}^{\rm HN}(\sE))$.
  \end{itemize}
  \end{thm}

\begin{rmk}\label{rmk3.3} In \cite[Theorem 1.3.4]{HL}, the proof of existence of
the filtration is given in terms of Gieseker stability. In
particular, ${\rm gr}_i^{\rm HN}(\sE)$ are Gieseker semi-stable,
thus they are $\mu$-semistable torsion free sheaves.
\end{rmk}

By using this unique filtration of $\sE$, we can introduce an
invariant ${\rm I}(\sE)$ of $\sE$, which we call the instability of
$\sE$. It is a rational number and measures how far is $\sE$ from
being semi-stable.

\begin{defn}\label{defn3.4} Let $\mu_{\rm max}(\sE)=\mu({\rm gr}_1^{\rm
HN}(\sE))$,
  $\mu_{\rm min}(\sE)=\mu({\rm gr}_{\ell+1}^{\rm HN}(\sE))$. Then the instability of $\sE$ is
  defined to be
  $${\rm I}(\sE):=\mu_{\rm max}(\sE)-\mu_{\rm min}(\sE).$$
\end{defn}

It is easy to see that a torsion free sheaf $\sE$ is semi-stable if
and only if ${\rm I}(\sE)=0$. We collect some elementary facts.

\begin{prop}\label{prop3.5} Let ${\rm HN}_{\bullet}(\sE)$ be the Harder-Narasimhan
filtration of length $\ell$ and $\mu_i=\mu({\rm gr}_i^{\rm
HN}(\sE))$ ($i=1,\,\ldots,\,\ell+1$). Then \begin{itemize}
\item[(1)]\,\,$\mu_{\rm max}(\sE/{\rm
HN}_i(\sE))=\mu_{i+1}$,\quad $\mu_{\rm min}({\rm HN}_i(\sE))=\mu_i$
\item[(2)] \,\,$\mu({\rm HN}_1(\sE))>\mu({\rm HN}_2(\sE))>\,\cdots\,> \mu({\rm
HN}_{\ell-1}(\sE))>\mu(\sE)$
\item[(3)]\,For any torsion free quotient $\sE\to \sQ\to 0$ and $\sE'\subset\sE$,
$$\mu(\sQ)\ge \mu_{\rm
min}(\sE),\quad \mu(\sE')\le\mu_{\rm max}(\sE)$$
\item[(4)]\,For any torsion free sheaves $\sF$, $\sE$, if $\mu_{\rm
min}(\sF)>\mu_{\rm max}(\sE)$, then
$${\rm Hom}(\sF,\,\sE)=0.$$
\end{itemize}
\end{prop}

\begin{proof} (1) follows the definition. (2) was proved in
\cite[Lemma 1.3.11]{HN} for curves, but the proof there works also
for higher dimensional varieties. The sub-sheaf case in (3) follows
from \cite[Lemma 1.3.5]{HL}. To see that $\mu(\sQ)\ge
\mu_{\min}(\sE)$, by Theorem \ref{thm3.2}, we can replace $\sQ$ by
the last grade quotient of ${\rm HN}_{\bullet}(\sQ)$, thus we can
assume that $\sQ$ is semi-stable. Then the quotient morphism induces
a non-trivial morphism ${\rm gr}_i^{\rm HN}(\sE)\to \sQ$. Thus
$\mu(\sQ)\ge\mu_i\ge\mu_{\rm min}(\sE)$. (4) follows from (3).
\end{proof}

In this section, we discuss the behavior of ${\rm I}(\sE)$ under the
Frobenius pull-back. We start it by introducing some discrete
invariants of a torsion free sheaf and its Frobenius pull-back. A
sub-sheaf $\sF\subset F^*\sE$ is called $\nabla_{\rm can}$-invariant
if $\nabla_{\rm can}(\sF)\subset \sF\otimes\Omega^1_X$, where
$\nabla_{\rm can}$ is the canonical connection on $F^*\sE$.

\begin{defn}\label{defn3.6} Let $\ell(\sE)=\ell$ be the length of
the Harder-Narasimhan filtration ${\rm HN}_{\bullet}(\sE)$ of $\sE$
and $s(X,\sE)$ be the number of $\nabla_{\rm can}$-invariant
sub-sheaves ${\rm HN}_i(F^*\sE)\subset F^*\sE$ that appears in ${\rm
HN}_{\bullet}(F^*\sE)$.
\end{defn}

Our goal is to bound ${\rm I}(F^*\sE)$ in terms of ${\rm I}(\sE)$,
$\ell(F^*\sE)$, $s(X,\sE)$ and some invariants of $X$. The lower
bound of ${\rm I}(F^*\sE)$
$${\rm I}(\sE)\le \frac{1}{p}\,\,{\rm I}(F^*\sE)$$
is trivial by using Proposition \ref{prop3.5} (3).

A upper bound of ${\rm I}(F^*\sE)$ was found in \cite[Theorem
3.1]{Sun} when $X$ is a curve of genus $g\ge 1$ and $\sE$ is
semi-stable (See also \cite{Sh}). One of the main observations in
the proof of \cite[Theorem 3.1]{Sun} is \ga{3.1}{{\rm
I}(F^*\sE)=\sum^{\ell}_{i=1}\{\mu_{\rm min}({\rm
HN}_i(F^*\sE))-\mu_{\rm max}(F^*\sE/{\rm HN}_i(F^*\sE))\}} where
$\ell=\ell(F^*\sE)$. Then, when $\sE$ is semi-stable, all of the
sub-sheaves ${\rm HN}_i(F^*\sE)$ ($1\le i\le\ell$) are not
$\nabla_{\rm can}$-invariant. Thus $\nabla_{\rm can}$ induces
nontrivial $\sO_X$-homomorphisms $${\rm HN}_i(F^*\sE))\to
\frac{F^*\sE}{{\rm HN}_i(F^*\sE)}\otimes\Omega^1_X\quad (1\le
i\le\ell)$$ which, by Proposition \ref{prop3.5}, imply
\ga{3.2}{\mu_{\rm min}({\rm HN}_i(F^*\sE))\le\mu_{\rm
max}(\frac{F^*\sE}{{\rm HN}_i(F^*\sE)}\otimes\Omega^1_X)\quad(1\le
i\le\ell).} When $\Omega^1_X$ has rank one, we have, for all $1\le
i\le\ell$, \ga{3.3}{\mu_{\rm max}(\frac{F^*\sE}{{\rm
HN}_i(F^*\sE)}\otimes\Omega^1_X)=\mu_{\rm max}(\frac{F^*\sE}{{\rm
HN}_i(F^*\sE)})+\mu(\Omega^1_X)} which implies immediately
$${\rm I}(F^*\sE)\le \ell\cdot(2g-2)\le({\rm rk}(\sE)-1)(2g-2).$$
In a more general version, we have
\begin{thm}\label{thm3.7} Let $X$ be a smooth projective curve of
genus $g\ge 1$ and $\sE$ a vector bundle on $X$. Let
$\ell(F^*\sE)=\ell$, $s(X,\sE)=s$. Then
$$p\cdot{\rm I}(\sE)\le {\rm I}(F^*\sE)\le (\ell-s)(2g-2)+p\cdot
s\cdot{\rm I}(\sE).$$
\end{thm}

\begin{proof} Let $S$ be the set of numbers $1\le i_k\le \ell$ such
that ${\rm HN}_{i_k}(F^*\sE)$ is a $\nabla_{\rm can}$-invariant
sub-sheaf of $F^*\sE$.  Let $\mu_i=\mu({\rm gr}^{\rm
HN}_i(F^*\sE))$, notice $\mu_{\rm max}(F^*\sE/{\rm
HN}_i(F^*\sE))=\mu_{i+1}$, $ \mu_{\rm min}({\rm
HN}_i(F^*\sE))=\mu_i$, we have
$$\aligned &{\rm I}(F^*\sE)=\mu_1-\mu_{\ell+1}=
\sum^{\ell}_{i=1}(\mu_i-\mu_{i+1})\\
&=\sum^{\ell}_{i=1}\{\mu_{\rm min}({\rm HN}_i(F^*\sE))-\mu_{\rm
max}(F^*\sE/{\rm HN}_i(F^*\sE))\}.
\endaligned$$
When $i\notin S$, ${\rm HN}_i(F^*\sE)$ is not $\nabla_{\rm
can}$-invariant, which means that
$${\rm
HN}_i(F^*\sE)\xrightarrow{\nabla_{\rm
can}}(F^*\sE)\otimes\Omega^1_X\to F^*\sE/{\rm
HN}_i(F^*\sE)\otimes\Omega^1_X$$ is a nontrivial
$\sO_X$-homomorphism. By Proposition \ref{prop3.5} (4), we have
$$\aligned\mu_{\rm min}({\rm
HN}_i(F^*\sE))&\le\mu_{\rm max}(F^*\sE/{\rm
HN}_i(F^*\sE)\otimes\Omega^1_X )\\&=\mu_{\rm max}(F^*\sE/{\rm
HN}_i(F^*\sE))+2g-2.\endaligned$$ Thus, for $i\notin S$, we have
$$\mu_{\rm min}({\rm HN}_i(F^*\sE))-\mu_{\rm max}(F^*\sE/{\rm
HN}_i(F^*\sE))\le 2g-2.$$ When $i\in S$, by Theorem \ref{thm2.2},
there is a sub-sheaf $\sE_i\subset \sE$ such that ${\rm
HN}_i(F^*\sE)=F^*\sE_i$ and $F^*\sE/{\rm
HN}_i(F^*\sE)=F^*(\sE/\sE_i)$. Then
$${\rm I}(F^*\sE)\le (\ell-s)(2g-2)+\sum_{i\in S}(\mu_{\rm min}(F^*\sE_i)-\mu_{\rm
max}(F^*(\sE/\sE_i))).$$ Notice $\mu_{\rm min}(F^*\sE_i)\le
\mu(F^*\sE_i)$, $\mu_{\rm max}(F^*(\sE/\sE_i))\ge
\mu(F^*(\sE/\sE_i))$ and $\mu(\sE_i)\le \mu_{\rm max}(\sE)$,
$\mu(\sE/\sE_i)\ge \mu_{\rm min}(\sE)$, we have
$$\mu_{\rm min}(F^*\sE_i)-\mu_{\rm
max}(F^*(\sE/\sE_i))\le p\,{\rm I}(\sE).$$ Thus
$$p\cdot{\rm I}(\sE)\le {\rm I}(F^*\sE)\le (\ell-s)(2g-2)+p\cdot
s\cdot{\rm I}(\sE).$$
\end{proof}

When ${\rm dim}(X)>1$ and $\sE$ is semi-stable, an upper bound of
${\rm I}(F^*\sE)$ was given in \cite[Corollary 6.2]{L} by A. Langer.
Before the discussion of his result, we make some remarks at first.
It is easy to see that all of the arguments above go through except
the equation \eqref{3.3} does not hold in general. Thus one can ask
the following question

\begin{question}\label{question3.8} What is the constant $a(\sE, X)$
such that $$\mu_{\rm max}(F^*\sE/{\rm HN}_i(F^*\sE)\otimes\Omega^1_X
)=\mu_{\rm max}(F^*\sE/{\rm HN}_i(F^*\sE))+a(\sE,X)\,\, ?$$ More
general, what is the  upper bound of
$$\mu_{\rm max}(\sE_1\otimes\sE_2)-\mu_{\rm max}(\sE_1)-\mu_{\rm
max}(\sE_2)$$ for any torsion free sheaves $\sE_1$ and $\sE_2$ ?
\end{question}

\begin{rmk}\label{rmk3.9} Let $a(\sE,X)$ be the constant in Question
\ref{question3.8}. Then, for any torsion free sheaf $\sE$ on a
smooth projective variety $X$, the proof of Theorem \ref{thm3.7}
implies the following inequalities
$$p\cdot{\rm I}(\sE)\le {\rm I}(F^*\sE)\le (\ell-s)\cdot a(\sE,X)+p\cdot
s\cdot{\rm I}(\sE)$$ where $\ell$ is the length of the
Harder-Narasimhan filtration ${\rm HN}_{\bullet}(F^*\sE)$ and $s$ is
the number of $\nabla_{\rm can}$-invariant sub-sheaves ${\rm
HN}_i(F^*\sE)$.
\end{rmk}

The difficult to answer Question \ref{question3.8} lies in the fact
that tensor product of two semi-stable sheaves may not be
semi-stable in the case of positive characteristic (such examples
are easy to construct, see Remark 5.10). However, the following
theorem was known by many peoples (see \cite[Theorem 6.1]{L}, where
it is referred to a special case of \cite[Theorem 3.23]{RR}).

\begin{thm}\label{thm3.10} A sheaf is called strongly semi-stable (resp. stable) if
its pullback by $k$-th power $F^k$ of Frobenius is semi-stable
(resp. stable) for any $k\ge 0$. Then a tensor product of two
strongly semi-stable sheaves is a strongly semi-stable sheaf.
\end{thm}

One of theorems proved by A. Langer in his celebrated paper \cite{L}
is the following

\begin{thm}\label{thm3.11} For any torsion free sheaf $\sE$, there
is an $k_0$ such that all of quotients ${\rm gr}_i^{\rm
HN}(F^{k*}\sE)$ in the Harder-Narasimhan  of $F^{k*}\sE$ are
strongly semi-stable whenever $k\ge k_0$.
\end{thm}

\begin{prop}\label{prop3.12} If all quotients ${\rm gr}_i^{\rm
HN}(\sE_1)$, ${\rm gr}_i^{\rm HN}(\sE_2)$ in Harder-Narasimhan
filtration of $\sE_1$ and $\sE_2$ are strongly semi-stable, then
$$\mu_{\rm max}(\sE_1\otimes\sE_2)\le \mu_{\rm
max}(\sE_1)+\mu_{\rm max}(\sE_2).$$ In particular, if all ${\rm
gr}_i^{\rm HN}(F^*\sE)$ are strongly semi-stable, then \ga{3.4}
{p\cdot{\rm I}(\sE)\le {\rm I}(F^*\sE)\le (\ell-s)\cdot \mu_{\rm
max}(\Omega^1_X)+p\cdot s\cdot{\rm I}(\sE)} where $\ell$ is the
length of the Harder-Narasimhan filtration ${\rm
HN}_{\bullet}(F^*\sE)$ and $s$ is the number of $\nabla_{\rm
can}$-invariant sub-sheaves ${\rm HN}_i(F^*\sE)$.
\end{prop}

\begin{proof} Since $\sE_1\otimes\sE_2$ has at most torsion of
dimension $n-2$, without loss of generality, we can assume that
$\sE_1\otimes\sE_2$ is torsion free. Let
$$\sF={\rm HN}_1(\sE_1\otimes\sE_2)\subset\sE_1\otimes\sE_2,\quad \mu(\sF)=\mu_{\rm
max}(\sE_1\otimes\sE_2).$$ By Theorem \ref{thm3.11}, there is a
$k_0$ such that for all $k\ge k_0$
$$\sF_k:={\rm HN}_1(F^{k*}(\sE_1\otimes\sE_2))\subset
F^{k*}(\sE_1\otimes\sE_2)=F^{k*}\sE_1\otimes F^{k*}\sE$$ are
strongly semi-stable. By Proposition \ref{prop3.5}, the nontrivial
homomorphism $(F^{k*}\sE_1)^{\vee}\otimes\sF_k\to F^{k*}\sE_2$
implies $$\mu_{\rm min}((F^{k*}\sE_1)^{\vee}\otimes\sF_k)\le
\mu_{\rm max}(F^{k*}\sE_2).$$ Since ${\rm gr}_i^{\rm HN}(\sE_1)$,
${\rm gr}_i^{\rm HN}(\sE_2)$, $\sF_k$ are strongly semi-stable, by
Theorem \ref{thm3.10}, we have $\mu(\sF)\le \mu_{\rm
max}(\sE_1)+\mu_{\rm max}(\sE_2)$.

To show \eqref{3.4}, it is enough to show \ga{3.5} {\mu_{\rm
min}({\rm HN}_i(F^*\sE))-\mu_{\rm max}(F^*\sE/{\rm HN}_i(F^*\sE))\le
\mu_{\max}(\Omega^1_X)} when ${\rm HN}_i(F^*\sE))$ is not
$\nabla_{\rm can}$-invariant. In this case, there is a nontrivial
homomorphism $T_X\to (F^*\sE/{\rm HN}_i(F^*\sE))\otimes {\rm
HN}_i(F^*\sE)^{\vee}$. Then
$$\mu_{\min}(T_X)\le \mu_{\max}(F^*\sE/{\rm
HN}_i(F^*\sE))+\mu_{\max}({\rm HN}_i(F^*\sE)^{\vee})$$ since all
${\rm gr}_i^{\rm HN}(F^*\sE)$ are strongly semi-stable.
\end{proof}

The inequality \eqref{3.4} has the following corollary, which was
proved by Mehta and Ramanathan (See \cite[Theorem 2.1]{MR}).

\begin{cor}\label{cor3.13} If $\mu_{\max}(\Omega^1_X)\le 0$, then
all semi-stable sheaves on $X$ are strongly semi-stable. If
$\mu_{\max}(\Omega^1_X)<0$, then all stable sheaves on $X$ are
strongly stable.
\end{cor}

\begin{proof} Let $\sE$ be a semi-stable sheaf of rank $r$ and
assume the corollary true for all semi-stable sheaves of rank
smaller than $r$. Then, if $F^*\sE$ is not semi-stable, all ${\rm
gr}_i^{\rm HN}(F^*\sE)$ are strongly semi-stable by the assumption.
Thus, by inequality \eqref{3.4}, $F^*\sE$ must be semi-stable.

If $\mu_{\max}(\Omega^1_X)<0$ and $\sE$ is stable, then for any
proper sub-sheaf $\sF\subset F^*\sE$, $\mu(\sF)\le\mu(F^*\sE)$. If
$\mu(\sF)=\mu(F^*\sE)$, then $\sF$ is not a pullback of a sub-sheaf
of $\sE$ since $\sE$ is stable. Thus the $\sO_X$-homomorphism
$$\sF\xrightarrow{\nabla_{\rm can}} F^*\sE\otimes\Omega^1_X\to
F^*\sE/\sF\otimes\Omega^1_X$$ is non-trivial, which implies
$\mu_{\max}(\Omega^1_X)\ge 0$ since $\sF$, $F^*\sE/\sF$ are strongly
semi-stable with the same slope.
\end{proof}

Now it becomes clear, since $p^{k-1}{\rm I}(F^*\sE)\le{\rm
I}(F^{k*}\sE)$, one can bound $$\frac{{\rm I}(F^{k*}\sE)}{p^k},\quad
k\ge k_0$$ where the difficult in Question \ref{question3.8}
vanishes by Proposition \ref{prop3.12}. Indeed, A. Langer made the
following definition in \cite{L}:
$$L_{\max}(\sE):=\lim_{k\to\infty}\frac{\mu_{\max}(F^{k*}\sE)}{p^k},\quad
L_{\min}(\sE):=\lim_{k\to\infty}\frac{\mu_{\min}(F^{k*}\sE)}{p^k}.$$
Then he proved the following (See \cite[Corollary 6.2]{L})

\begin{thm}\label{thm3.13} Let $\sE$ be a semi-stable torsion free
sheaf. Then
$$L_{\max}(\sE)-L_{\min}(\sE)\le
\frac{{\rm
rk}(\sE)-1}{p}\cdot{\max}\{\,0,\,\,L_{\max}(\Omega^1_X)\,\}$$ In
particular, \,\,${\rm I}(F^*\sE)\le ({\rm rk}(\sE)-1)\cdot
{\max}\{\,0,\,\,L_{\max}(\Omega^1_X)\,\}$.
\end{thm}

For a torsion free sheaf $\sE$ of rank $r$, by Theorem
\ref{thm3.11}, there is a $k_0$ such that all of quotients ${\rm
gr}_i^{\rm HN}(F^{k*}\sE)$ in the Harder-Narasimhan  of $F^{k*}\sE$
are strongly semi-stable whenever $k\ge k_0$. We choose $k_0$ to be
the minimal integer such that all quotients ${\rm gr}_i^{\rm
HN}(F^{k_0*}\sE)$ in
$$ 0\subset{\rm
HN}_1(F^{k_0*}\sE)\subset\,\cdots\,\subset {\rm
HN}_{\ell}(F^{k_0*}\sE)\subset {\rm
HN}_{\ell+1}(F^{k_0*}\sE)=F^{k_0*}\sE$$ are strongly semi-stable.
For each ${\rm HN}_i(F^{k_0*}\sE)$ ($1\le i\le\ell$), there is a
$0\le k_i\le k_0$ and a sub-sheaf $\sE_i\subset F^{k_i*}\sE$ such
that \ga{3.6} {{\rm HN}_i(F^{k_0*}\sE)=F^{k_0-k_i*}\sE_i,\quad
\nabla_{\rm can}(\sE_i)\nsubseteq \sE_i\otimes\Omega^1_X \,\, {\rm
if}\,\, k_i>0.} Let $S=\{\,1\le i\le k_0\,\,|\,\,k_i=0\,\,\}$. Then,
for $i\in S$, \ga{3.7}{\mu_{\rm min}({\rm
HN}_i(F^{k_0*}\sE))-\mu_{\rm max}(\frac{F^{k_0*}\sE}{{\rm
HN}_i(F^{k_0*}\sE)})\le p^{k_0}{\rm I}(\sE).} For $i\notin S$, there
is a nontrivial $\sO_X$-homomorphism
$${\rm HN}_i(F^{k_0*}\sE)\to \frac{F^{k_0*}\sE}{{\rm
HN}_i(F^{k_0*}\sE)}\otimes F^{k_0-k_i*}\Omega^1_X$$ which is the
pullback of $\sE_i\xrightarrow{\nabla_{\rm
can}}F^{k_i*}\sE\otimes\Omega^1_X\to\frac{F^{k_i*}\sE}{\sE_i}\otimes\Omega^1_X$.
Thus \ga{3.8}{\mu_{\rm min}({\rm HN}_i(F^{k_0*}\sE))-\mu_{\rm
max}(\frac{F^{k_0*}\sE}{{\rm HN}_i(F^{k_0*}\sE)})\le
\mu_{\max}(F^{k_0-k_i*}\Omega^1_X).} Notice that
$p^{k_i}\mu_{\max}(F^{k_0-k_i*}\Omega^1_X)\le\mu_{\max}(F^{k_0*}\Omega^1_X)$,
we have \ga{3.9}{{\rm I}(F^{k_0*}\sE)\le
\frac{\ell-s}{p}\mu_{\max}(F^{k_0*}\Omega^1_X)+s\cdot p^{k_0}{\rm
I}(\sE)} where $s=|S|$ is number of elements in $S$. Since, for any
$k\ge k_0$, ${\rm I}(F^{k*}\sE)=p^{k-k_0}{\rm I}(F^{k_0*}\sE)$, we
have \ga{3.10} {\frac{{\rm I}(F^{k*}\sE)}{p^k}\le
\frac{\ell-s}{p}\cdot\frac{\mu_{\max}(F^{k*}\Omega^1_X)}{p^k}+s\cdot{\rm
I}(\sE).} By Corollary \ref{cor3.13}, to study ${\rm I}(F^*\sE)$, it
is enough to consider varieties $X$ with $\mu_{\max}(\Omega^1_X)>
0$. Then we can formulate above discussions as
\begin{thm}\label{thm3.15} Let $X$ be a smooth projective variety of $\mu_{\max}(\Omega^1_X)>0$.
Then, for any torsion free sheaf $\sE$ of rank $r$, we have
$$L_{\max}(\sE)-L_{\min}(\sE)\le\frac{\ell-s}{p}\cdot L_{\max}(\Omega^1_X)+s\cdot{\rm
I}(\sE).$$ In particular, ${\rm I}(F^*\sE)\le
(r-1)(L_{\max}(\Omega^1_X)+{\rm I}(\sE))$.
\end{thm}

\begin{rmk}\label{rmk3.16} It is clear that $L_{\max}(\sE)-L_{\min}(\sE)=\frac{{\rm
I}(F^{k_0*}\sE)}{p^{k_0}}$ and
$${\rm I}(F^*\sE)\le (\ell-s)\cdot\frac{\mu_{\max}(F^{k_0*}\Omega^1_X)}{p^{k_0}}+s\cdot{\rm
I}(\sE).$$ One may make the following conjecture that \ga{3.11}{{\rm
I}(F^*\sE)\le (r-1)\mu_{\max}(\Omega^1_X)+(r-1){\rm I}(\sE).}
\end{rmk}

\section{Instability of Frobenius direct images}

In this section, we study the instability of direct image $F_*W$ for
a torsion free sheaf $W$ on $X$. For example, is $F_*W$ semi-stable
when $W$ is semi-stable ?  Compare with the case of characteristic
zero, for a Galois $G$-cover $\pi:Y\to X$, the locally free sheaf
$\pi_*\sO_Y$ is not semi-stable if $\pi$ is not {\`e}tale. However,
if $\pi$ is {\`e}tale, then $\pi_*W$ is semi-stable whenever $W$ is
semi-stable. The proof of this fact is based on a decomposition
\ga{4.1}{\pi^*(\pi_*W)=\bigoplus_{\sigma\in G}W^{\sigma}.} To
imitate this idea, we need a similar "decomposition" of
$V=F^*(F_*W)$ for $F:X\to X_1$. In general, we can not expect to
have a real decomposition of $V=F^*(F_*W)$. Instead of, we will have
a filtration \ga{4.2}{0=V_{n(p-1)+1}\subset
V_{n(p-1)}\subset\cdots\subset V_1\subset V_0=V} such that
$V_{\ell}/V_{\ell+1}\cong W\otimes_{\mathcal{O}_X}{\rm
T}^{\ell}(\Omega^1_X)$.

The filtration \eqref{4.2} was defined and studied in \cite{JRXY}
for curves. Its definition can be generalized straightforwardly by
using the canonical connection $\nabla_{\rm can}:V\to
V\otimes\Omega^1_X$. The study of its graded quotients are much
involved (cf. \cite[Section 3]{Su}).

\begin{defn}\label{defn4.1} Let $V_0:=V=F^*(F_*W)$,
$V_1=\ker(F^*(F_*W)\surj W)$ \ga{4.3}
{V_{\ell+1}:=\ker\{V_{\ell}\xrightarrow{\nabla} V\otimes_{\sO_X}
\Omega^1_X\to (V/V_{\ell})\otimes_{\sO_X}\Omega^1_X\}} where
$\nabla:=\nabla_{\rm can}$ is the canonical connection.
\end{defn}

In order to describe the filtration, we recall a ${\rm
GL}(n)$-representation ${\rm T}^{\ell}(V)\subset V^{\otimes\ell}$
where $V$ is the standard representation of ${\rm GL}(n)$. Let ${\rm
S}_{\ell}$ be the symmetric group of $\ell$ elements with the action
on $V^{\otimes\ell}$ by $(v_1\otimes\cdots\otimes
v_{\ell})\cdot\sigma=v_{\sigma(1)}\otimes\cdots\otimes
v_{\sigma(\ell)}$ for $v_i\in V$ and $\sigma\in{\rm S}_{\ell}$. Let
$e_1,\,\ldots,\,e_n$ be a basis of $V$, for $k_i\ge 0$ with
$k_1+\cdots+k_n=\ell$ define \ga{4.4}
{v(k_1,\ldots,k_n)=\sum_{\sigma\in{\rm S}_{\ell}}(e_1^{\otimes
k_1}\otimes\cdots\otimes e_n^{\otimes k_n})\cdot\sigma }

\begin{defn}\label{defn4.2} Let ${\rm T}^{\ell}(V)\subset V^{\otimes\ell}$ be
the linear subspace generated by all vectors $v(k_1,\ldots,k_n)$ for
all $k_i\ge 0$ satisfying $k_1+\cdots+k_n=\ell$. It is a
representation of ${\rm GL}(V)$. If $\sV$ is a vector bundle of rank
$n$,  the subbundle ${\rm T}^{\ell}(\sV)\subset \sV^{\otimes\ell}$
is defined to be the associated bundle of the frame bundle of $\sV$
(which is a principal ${\rm GL}(n)$-bundle) through the
representation ${\rm T}^{\ell}(V)$.
\end{defn}

Then the following theorem was proved in \cite[Theorem 3.7]{Su}

\begin{thm}\label{thm4.3} The filtration defined in Definition
\ref{defn4.1} is \ga{4.5} {0=V_{n(p-1)+1}\subset
V_{n(p-1)}\subset\cdots\subset V_1\subset V_0=V=F^*(F_*W)} which has
the following properties
\begin{itemize}\item[(i)]$\nabla(V_{\ell+1})\subset V_{\ell}\otimes\Omega^1_X$
for $\ell\ge 1$, and $V_0/V_1\cong W$.
\item[(ii)]
$V_{\ell}/V_{\ell+1}\xrightarrow{\nabla}(V_{\ell-1}/V_{\ell})\otimes\Omega^1_X$
are injective for $1\le \ell\le n(p-1)$, which induced isomorphisms
$$\nabla^{\ell}: V_{\ell}/V_{\ell+1}\cong W\otimes_{\sO_X}{\rm
T}^{\ell}(\Omega^1_X),\quad 0\le \ell\le n(p-1).$$ The vector bundle
${\rm T}^{\ell}(\Omega^1_X)$ is suited in the exact sequence
$$\aligned&0\to{\rm Sym}^{\ell-\ell(p)\cdot p}(\Omega^1_X)\otimes
F^*\Omega_X^{\ell(p)}\xrightarrow{\phi}{\rm
Sym}^{\ell-(\ell(p)-1)\cdot p}(\Omega^1_X)\otimes
F^*\Omega_X^{\ell(p)-1}\\&\to\cdots\to {\rm Sym}^{\ell-q\cdot
p}(\Omega^1_X)\otimes F^*\Omega_X^q\xrightarrow{\phi}{\rm
Sym}^{\ell-(q-1)\cdot p}(\Omega^1_X)\otimes
F^*\Omega_X^{q-1}\\&\to\cdots\to{\rm
Sym}^{\ell-p}(\Omega^1_X)\otimes F^*\Omega^1_X\xrightarrow{\phi}{\rm
Sym}^{\ell}(\Omega^1_X)\to{\rm T}^{\ell}(\Omega_X^1)\to 0
\endaligned$$
where $\ell(p)\ge 0$ is the integer such that $\ell-\ell(p)\cdot
p<p$.
\end{itemize}
\end{thm}

It is this filtration that we used in \cite{Su} to find a upper
bound of ${\rm I}(F_*W)$. To state the results, let $X$ be an
irreducible smooth projective variety of dimension $n$ over an
algebraically closed field $k$ with ${\rm char}(k)=p>0$. For any
torsion free sheaf $W$ on $X$, let
$${\rm I}(W,X)={\rm max}\{{\rm I}(W\otimes{\rm
T}^{\ell}(\Omega^1_X))\,|\,\, 0\le \ell\le n(p-1)\,\}$$ be the
maximal value of instabilities ${\rm I}(W\otimes{\rm
T}^{\ell}(\Omega^1_X))$. Then we have

\begin{thm}\label{thm4.4} When
$K_X\cdot{\rm H}^{n-1}\ge 0$, we have, for any $\sE\subset F_*W$,
\ga{4.6}{\mu(F_*W)-\mu(\sE)\ge -\frac{{\rm I}(W,X)}{p}.} In
particular, if $W\otimes{\rm T}^{\ell}(\Omega^1_X)$, $0\le \ell\le
n(p-1)$, are semistable, then $F_*W$ is semistable. Moreover, if
$K_X\cdot{\rm H}^{n-1}>0$, the stability of the bundles
$W\otimes{\rm T}^{\ell}(\Omega^1_X)$, $0\le \ell\le n(p-1)$, implies
the stability of $F_*W$.
\end{thm}

\begin{cor}\label{cor4.5} Let $X$ be a smooth projective variety of ${\rm dim}(X)=n$, whose
canonical divisor $K_X$ satisfies $K_X\cdot {\rm H}^{n-1}\ge 0$.
Then $${\rm I}(W)\le{\rm I}(F_*W)\le p^{n-1}{\rm rk}(W)\,{\rm
I}(W,X).$$
\end{cor}

\begin{proof} The lower bound is trivial, the upper bound is Theorem \ref{thm4.4} plus the following
trivial remark: For any vector bundle $E$, if there is a constant
$\lambda$ satisfying $\mu(E')-\mu(E)\le \lambda$ for any $E'\subset
E$. Then ${\rm I}(E)\le {\rm rk}(E)\lambda$.
\end{proof}

When ${\rm dim}(X)=1$, we have the following corollary, which was
proved in \cite{LP} when $W$ is a line bundle. The case that
semi-stability of $W$ implies semi-stability of $F_*W$ was also
proved in \cite{MP} by a different method. However, the method in
\cite{MP} was not able to prove the case that stability of $W$
implies stability of $F_*W$.

\begin{cor}\label{cor4.6}
When $g\ge 1$, $F_*(W)$ is semi-stable if and only if $W$ is
semi-stable. Moreover, if $g\ge 2$, then $F_*(W)$ is stable if and
only if $W$ is stable.
\end{cor}

\begin{proof} When ${\rm dim}(X)=1$, $W\otimes {\rm
T}^{\ell}(\Omega^1_X)=W\otimes {\Omega^1_X}^{\otimes\ell}$ is
semi-stable (resp. stable) whenever $W$ is semi-stable (resp.
stable). Thus $F_*W$ is semi-stable (resp. stable).

\end{proof}

Let $\sE\subset F_*W$ be a nontrivial subsheaf, the canonical
filtration \eqref{4.5} induces the filtration (we assume $V_m\cap
F^*\sE\neq 0$) \ga{4.7}{0\subset V_m\cap
F^*\sE\subset\,\cdots\,\subset V_1\cap F^*\sE\subset V_0\cap
F^*\sE=F^*\sE.} Let
$$\sF_{\ell}:=\frac{V_{\ell}\cap F^*\sE}{V_{\ell+1}\cap
F^*\sE}\subset\frac{V_{\ell}}{V_{\ell+1}}, \qquad r_{\ell}={\rm
rk}(\sF_{\ell}).$$ Then $\mu(F^*\sE)=\frac{1}{{\rm
rk}(\sE)}\sum_{\ell=0}^mr_{\ell}\cdot\mu(\sF_{\ell})$ and
\ga{4.8}{\mu(\sE)-\mu(F_*W)=\frac{1}{p\cdot{\rm
rk}(\sE)}\sum^m_{\ell=0}r_{\ell}\left(\mu(\sF_{\ell})-\mu(F^*F_*W)\right).}

\begin{lem}\label{lem4.7} With the same notation in Theorem
\ref{thm4.3}, we have \ga{4.9} {\mu(F^*F_*W)=p\cdot\mu(F_*W)=
\frac{p-1}{2}K_X\cdot{\rm
H}^{n-1}+\mu(W), \\
 \mu(V_{\ell}/V_{\ell+1})=\mu(W\otimes {\rm
T^{\ell}}(\Omega^1_X))=\frac{\ell}{n}K_X\cdot{\rm H}^{n-1}+\mu(W).
\notag}
\end{lem}
By using above lemma (See \cite{Su} for the proof), we have
\ga{4.10}{\mu(\sE)-\mu(F_*W)=
\sum^m_{\ell=0}r_{\ell}\frac{\mu(\sF_{\ell})
-\mu(\frac{V_{\ell}}{V_{\ell+1}})}{p\cdot{\rm rk}(\sE)}\\-
\notag\frac{\mu(\Omega^1_X)}{p\cdot{\rm
rk}(\sE)}\sum^m_{\ell=0}(\frac{n(p-1)}{2}-\ell)r_{\ell}} It is clear
that $\mu(\sF_{\ell}) -\mu(V_{\ell}/V_{\ell+1})\le{\rm
I}(V_{\ell}/V_{\ell+1})={\rm I}(W\otimes {\rm
T}^{\ell}(\Omega^1_X))$. Thus the proof of Theorem \ref{thm4.4} will
be completed if one can prove
\begin{lem}\label{lem4.8} The ranks $r_{\ell}$ of $\sF_{\ell}\subset
V_{\ell}/V_{\ell+1}$ ($0\le\ell\le m$) satisfy
$$\sum^m_{\ell=0}(\frac{n(p-1)}{2}-\ell)r_{\ell}\ge 0.$$
\end{lem}

When $m\le\frac{n(p-1)}{2}$, the lemma is clear. In fact, we have
\ga{4.11}
{\sum^m_{\ell=0}(\frac{n(p-1)}{2}-\ell)r_{\ell}\ge\frac{n(p-1)}{2}r_0\ge\frac{n(p-1)}{2}.}
When $m>\frac{n(p-1)}{2}$, we can write \ga{4.12}
{\sum_{\ell=0}^m(\frac{n(p-1)}{2}-\ell)r_{\ell}=\sum^{n(p-1)}_{\ell=m+1}
(\ell-\frac{n(p-1)}{2})
r_{n(p-1)-\ell}\\+\sum^m_{\ell\,>\frac{n(p-1)}{2}}(\ell-\frac{n(p-1)}{2})
(r_{n(p-1)-\ell}-r_{\ell}).\notag} The numbers $r_{\ell}$
($0\le\ell\le m$) are related by the following fact that
$V_{\ell}/V_{\ell+1}\xrightarrow{\nabla}(V_{\ell-1}/V_{\ell})\otimes\Omega^1_X$
induce injective $\sO_X$-homomorphisms \ga{4.13}
{\sF_{\ell}\xrightarrow{\nabla}\sF_{\ell-1}\otimes\Omega^1_X\quad
(1\le\ell\le m).} It is based on this fact that we proved in
\cite{Su} the following inequalities
$$r_{n(p-1)-\ell}-r_{\ell}\ge 0 \qquad (\ell>\frac{n(p-1)}{2})$$
which complete the proof of Lemma \ref{lem4.8}.

The proof of Theorem \ref{thm4.4} has more implications than the
theorem itself. Recall that the sheaf $B^1_X$ of locally exact
differential forms on $X$ is defined by exact sequence \ga{4.14}
{0\to \sO_X\to F_*\sO_X \to B^1_X\to 0.}

\begin{thm}\label{thm4.9} Let $\sL$ be a torsion free sheaf of rank
$1$. Then, for any nontrivial $\sE\subset F_*\sL$ with ${\rm
rk}(\sE)<{\rm rk}(F_*\sL)$, we have \ga{4.15}
{\mu(\sE)-\mu(F_*\sL)\le\frac{{\rm I}(\sL,X)}{p}-
\frac{\mu(\Omega^1_X)}{p\cdot{\rm rk}(\sE)}\cdot\frac{n(p-1)}{2}.}
In particular, when $\mu(\Omega_X^1)>0$ and ${\rm
T}^{\ell}(\Omega^1_X)$ ($1\le\ell<n(p-1)$) are semi-stable, then
$F_*\sL$ and $B^1_X$ are stable.
\end{thm}

\begin{proof} Since $\mu(\sF_{\ell})
-\mu(V_{\ell}/V_{\ell+1})\le {\rm I}(\sL\otimes{\rm
T}^{\ell}(\Omega^1_X))={\rm I}({\rm T}^{\ell}(\Omega^1_X))$ and
${\rm I}(\sL,X)={\max}\{\,\,{\rm I}({\rm
T}^{\ell}(\Omega^1_X))\,\,|\,\,1\le\ell<n(p-1)\,\,\}$, by
\eqref{4.10}, we only have to show
$$\sum^m_{\ell=0}(\frac{n(p-1)}{2}-\ell)r_{\ell}\ge\frac{n(p-1)}{2}.$$
From \eqref{4.11} and \eqref{4.12}, we have
$$\sum^m_{\ell=0}(\frac{n(p-1)}{2}-\ell)r_{\ell}\ge\frac{n(p-1)}{2}r_0\quad
\text{if\, $m\neq n(p-1)$.}$$ Thus it is enough to show $m\neq
n(p-1)$ when ${\rm rk}(\sE)<{\rm rk}(F_*\sL)=p^n$. More general, we
can show the following inequality \ga{4.16} {r_{\ell}\ge
r_{n(p-1)}\cdot {\rm rk}({\rm T}^{n(p-1)-\ell}(\Omega^1_X))\quad
\text{when $m=n(p-1)$},} which implies the following inequality
$${\rm rk}(\sE)=\sum^m_{\ell=0}r_{\ell}\ge
r_{n(p-1)}\sum^m_{\ell=0}{\rm rk}({\rm
T}^{n(p-1)-\ell}(\Omega^1_X))=r_{n(p-1)}\cdot p^n$$ if $m=n(p-1)$.
Thus $m\neq n(p-1)$ when ${\rm rk}(\sE)<p^n$.

To show \eqref{4.16} is a local problem. Let $K=K(X)$ be the
function field of $X$ and consider the $K$-algebra
$$R=\frac{K[\alpha_1,\cdots,\alpha_n]}{(\alpha_1^p,\cdots,\alpha_n^p)}=\bigoplus^{n(p-1)}_{\ell=0}R^{\ell},$$
where $R^{\ell}$ is the $K$-linear space generated by
$$\{\,\alpha_1^{k_1}\cdots \alpha_n^{k_n}\,|\,k_1+\cdots+k_n=\ell,\quad 0\le k_i\le
p-1\,\}.$$ The quotients in the filtration \eqref{4.5} can be
described locally
$$V_{\ell}/V_{\ell+1}=W\otimes_K R^{\ell}$$
as $K$-vector spaces. If $K=k(x_1,..., x_n)$, then the homomorphism
$$\nabla: W\otimes_K R^{\ell}\to W\otimes_KR^{\ell-1}\otimes_K\Omega^1_{K/k}$$
in Theorem \ref{thm4.3} (ii) is locally the $k$-linear homomorphism
defined by
$$\nabla(w\otimes\alpha_1^{k_1}\cdots\alpha_n^{k_n})=-w\otimes\sum^n_{i=1}k_i
(\alpha_1^{k_1}\cdots\alpha_i^{k_i-1}\cdots\alpha_n^{k_n})
\otimes_K{\rm d}x_i.$$ Then the fact that
$\sF_{\ell}\xrightarrow{\nabla}\sF_{\ell-1}\otimes \Omega^1_X$ for
$\sF_{\ell}\subset W\otimes R^{\ell}$ is equivalent to
\ga{4.17}{\forall\,\,\, \sum_jw_j\otimes f_j \in
\sF_{\ell}\,\,\Rightarrow\,\,\sum_jw_j\otimes\frac{\partial
f_j}{\partial\alpha_i}\, \in\,\sF_{\ell-1}\quad (1\le i\le n).}

The polynomial ring ${\rm P}=
K[\partial_{\alpha_1},\cdots,\partial_{\alpha_n}]$ acts on $R$
through partial derivations, which induces a ${\rm D}$-module
structure on $R$, where
$${\rm D}=\frac{K[\partial_{\alpha_1},\cdots,\partial_{\alpha_n}]}{(\partial_{\alpha_1}^p,\cdots,\partial_{\alpha_n}^p)}
=\bigoplus^{n(p-1)}_{\ell=0}{\rm D}_{\ell}$$ and ${\rm D}_{\ell}$ is
the linear space of degree $\ell$ homogeneous elements. In
particular, $W\otimes R$ has the induced ${\rm D}$-module structure
with ${\rm D}$ acts on $W$ trivially. Use this notation,
\eqref{4.17} is equivalent to ${\rm D}_1\cdot \sF_{\ell}\subset
\sF_{\ell-1}$.

Since $R^{n(p-1)}$ is of dimension $1$, for any subspace
$$\sF_{n(p-1)}\subset W\otimes R^{n(p-1)},$$ there is a subspace
$W'\subset W$ of dimension $r_{n(p-1)}$ such that
$$\sF_{n(p-1)}=W'\otimes R^{n(p-1)}.$$
Thus ${\rm D}_{\ell}\cdot\sF_{n(p-1)}=W'\otimes {\rm D}_{\ell}\cdot
R^{n(p-1)}=W'\otimes R^{n(p-1)-\ell}\subset\sF_{n(p-1)-\ell}\,$ for
all $0\le\ell\le n(p-1)$, which proves \eqref{4.16}.

If ${\rm T}^{\ell}(\Omega^1_X)$ ($1\le\ell<n(p-1)$) are semi-stable,
then ${\rm I}(\sL,X)=0$ and \ga{4.18} {\mu(\sE)-\mu(F_*\sL)\le -
\frac{\mu(\Omega^1_X)}{p\cdot{\rm rk}(\sE)}\cdot\frac{n(p-1)}{2},}
which implies clearly the stability of $F_*\sL$ if
$\mu(\Omega^1_X)>0$.

To show that \eqref{4.18} implies the stability of $B^1_X$, for any
nontrivial subsheaf $B'\subset B^1_X$ of rank $r<{\rm rk}(B^1_X)$,
let $\sE\subset F_*\sO_X$ be the subsheaf of rank $r+1$ such that we
have exact sequence
$$0\to \sO_X\to\sE\to B'\to 0.$$
Substitute \eqref{4.18} to
$\mu(B')-\mu(B^1_X)=\frac{r+1}{r}\mu(\sE)-\frac{p^n}{p^n-1}\mu(F_*\sO_X)$,
we have \ga{4.19}{\mu(B')-\mu(B^1_X)\le
\frac{p^n-1-r}{r(p^n-1)}\mu(F_*\sO_X)-\frac{n(p-1)}{2rp}\mu(\Omega^1_X).}
By \eqref{4.9} in Lemma \ref{lem4.7}, we have
$\mu(F_*\sO_X)=\frac{n(p-1)}{2p}\mu(\Omega^1_X)$. Thus
\ga{4.20}{\mu(B')-\mu(B^1_X)\le
-\frac{\mu(\Omega^1_X)}{p\cdot(p^n-1)}\cdot\frac{n(p-1)}{2}.}
\end{proof}

\begin{rmk}\label{rmk4.10} When ${\rm dim}(X)=1$,
the quotients $V_{\ell}/V_{\ell+1}=\sL\otimes\omega_X^{\ell}$ are
line bundles and thus $r_{\ell}=1$ ($0\le \ell\le m$) in
\eqref{4.10}. Then we can rewrite \eqref{4.10} (notice ${\rm
rk}(\sE)=m+1$):
$$\mu(\sE)-\mu(F_*\sL)=\sum^m_{\ell=0}\frac{\mu(\sF_{\ell})-\mu(\frac{V_{\ell}}{V_{\ell+1}})}{p\cdot{\rm rk}(\sE)}
-\frac{(p-{\rm rk}(\sE))(g-1)}{p},$$ which impiles the following
stronger inequality
$$\mu(\sE)-\mu(F_*\sL)\le-\frac{(p-{\rm rk}(\sE))(g-1)}{p}$$
and the equality holds if and only if
$\sF_{\ell}=V_{\ell}/V_{\ell+1}$. Thus
$$\mu(B')-\mu(B^1_X)\le -\frac{p-1-{\rm rk}(B')}{p}(g-1).$$
When $X$ is a curve of genus $g\ge 2$, the stability of $F_*\sL$ was
proved in \cite{LP}, the semi-stability of $B^1_X$ was proved by M.
Raynaud in \cite{Ra}, its stability, which is related with a
question of M. Raynaud in \cite{R}, was proved by K. Joshi in
\cite{J}. When $X$ is a surface with $\mu(\Omega^1_X)>0$, if
$\Omega^1_X$ is semi-stable (which implies that ${\rm
T}^{\ell}(\Omega^1_X)$ ($1\le\ell\le 2(p-1)$ are semi-stable), thus
$F_*\sL$ and $B^1_X$ are stable. The semi-stability of $B^1_X$ was
proved by Y. Kitadai and H. Sumihiro in \cite{Kit}.
\end{rmk}

In the proof of Theorem \ref{thm4.9}, for a sub-sheaf $\sE\subset
F_*W$, we see that \ga{4.21}{\mu(\sE)-\mu(F_*W)\le
\sum^m_{\ell=0}r_{\ell}\frac{{\rm I}(W\otimes{\rm
T}^{\ell}(\Omega^1_X))}{p\cdot{\rm rk}(\sE)}-
\frac{\mu(\Omega^1_X)}{p\cdot{\rm rk}(\sE)}\frac{n(p-1)}{2}} if
$m\neq n(p-1)$. Otherwise there is a sub-sheaf $W'\subset W$ of rank
$r_{n(p-1)}$ such that $\sF_{n(p-1)}=W'\otimes{\rm
T}^{n(p-1)}(\Omega^1_X)$ and $W'\otimes{\rm
T}^{\ell}(\Omega^1_X)\subset\sF_{\ell}$. Let
$$0\to W'\otimes{\rm
T}^{\ell}(\Omega^1_X)\to\sF_{\ell}\to\sF_{\ell}'\to 0$$  be the
induced exact sequence with $\sF_{\ell}'\subset W/W'\otimes{\rm
T}^{\ell}(\Omega^1_X)$. Then
$$\aligned\mu(\sF_{\ell})-\mu(\frac{V_{\ell}}{V_{\ell+1}})\le&
\frac{r_{n(p-1)}({\rm
rk}(\frac{V_{\ell}}{V_{\ell+1}})-r_{\ell})}{r_{\ell}\cdot{\rm
rk}(W)}(\mu(W')-\mu(W/W'))\\&+\frac{r_{\ell}'}{r_{\ell}}\cdot{\rm
I}(W/W'\otimes{\rm T}^{\ell}(\Omega^1_X))\endaligned$$ where
$r_{\ell}':={\rm rk}(\sF_{\ell}')$. Substitute it to the equality
\eqref{4.10}, we have \ga{4.22} {\mu(\sE)-\mu(F_*W)\le
\sum^{n(p-1)}_{\ell=0}r_{\ell}'\cdot\frac{{\rm I}(W/W'\otimes{\rm
T}^{\ell}(\Omega^1_X))}{p\cdot{\rm rk}(\sE)}+\\
\notag\frac{r_{n(p-1)}({\rm rk}(F_*W)-{\rm rk}(\sE))}{p\cdot{\rm
rk}(\sE)\cdot{\rm rk}(W)}(\mu(W')-\mu(W/W')).}

In the case of positive characteristic, it is well-known that tensor
product of two semi-stable sheaves may not be semi-stable. Thus,
even if $W$ and ${\rm T}^{\ell}(\Omega^1_X)$ are semi-stable,
Theorem \ref{thm4.4} does not imply the semi-stability of $F_*W$.
However the inequalities \eqref{4.21} and \eqref{4.22} indicate that
it may be possible in some special cases that semi-stability of $W$
and ${\rm T}^{\ell}(\Omega^1_X)$ can imply the semi-stability of
$F_*W$. As an example, we prove a slightly generalized version of
\cite[Theorem 3.1]{Kit}.

\begin{thm}\label{thm4.11} Let $X$ be a smooth projective surface
with $\mu(\Omega^1_X)>0$.  Assume that $\Omega^1_X$ is semi-stable.
Then $F_*(\sL\otimes\Omega^1_X)$ is semi-stable for any line bundle
$\sL$ on $X$. Moreover, if $\Omega^1_X$ is stable, then
$F_*(\sL\otimes\Omega^1_X)$ is stable.
\end{thm}

\begin{proof} When ${\rm dim}(X)=2$, we have (cf. Proposition 3.5 of \cite{Su})
$${\rm T}^{\ell}(\Omega^1_X)=\left\{
\begin{array}{llll}{\rm Sym}^{\ell}(\Omega^1_X)  &\mbox{when $\ell<p$}\\
{\rm Sym}^{2(p-1)-\ell}(\Omega^1_X)\otimes\omega_X^{\ell-(p-1)}
&\mbox{when $\ell\ge p$}
\end{array}\right.$$
where $\omega_X=\sO_X(K_X)$ is the canonical line bundle of $X$.
Thus ${\rm T}^{\ell}(\Omega^1_X)$ are semi-stable whenever
$\Omega^1_X$ is semi-stable.

For any nontrivial sub-sheaf $\sE\subset F_*(\sL\otimes\Omega^1_X)$,
consider the induced filtration $0\subset V_m\cap
F^*\sE\subset\,\cdots\,\subset V_1\cap F^*\sE\subset V_0\cap
F^*\sE=F^*\sE$ and $$\sF_{\ell}:=\frac{V_{\ell}\cap
F^*\sE}{V_{\ell+1}\cap F^*\sE}\subset\frac{V_{\ell}}{V_{\ell+1}},
\qquad r_{\ell}={\rm rk}(\sF_{\ell}).$$ If $m=2(p-1)$, by using
\eqref{4.22} for $W=\Omega^1_X$, we have
$$\mu(\sE)-\mu(F_*W)\le 0.$$
If $W=\Omega^1_X$ is stable, then $\mu(W')-\mu(W/W')<0$ in
\eqref{4.22} and $$\mu(\sE)-\mu(F_*W)<0.$$ If $m\neq 2(p-1)$, we
have
$$\mu(\sE)-\mu(F_*W)\le\sum^m_{\ell=0}r_{\ell}\frac{\mu(\sF_{\ell})
-\mu(\frac{V_{\ell}}{V_{\ell+1}})}{p\cdot{\rm
rk}(\sE)}-\frac{\mu(\Omega^1_X)}{p\cdot{\rm rk}(\sE)}\cdot(p-1).$$
On the other hand, by a theorem of Ilangovan-Mehta-Parameswaran (cf.
Section 6 of \cite{L} for the precise statement): If $E_1$, $E_2$
are semi-stable bundles with ${\rm rk}(E_1)+{\rm rk}(E_2)\le p+1$,
then $E_1\otimes E_2$ is semi-stable. We see that
$V_{\ell}/V_{\ell+1}=\sL\otimes\Omega^1_X\otimes{\rm
T}^{\ell}(\Omega^1_X)$ are semi-stable except that
$$V_{p-1}/V_p=\sL\otimes\Omega^1_X\otimes{\rm
Sym}^{p-1}(\Omega^1_X)$$ may not be semi-stable. Thus we have
$$\mu(\sE)-\mu(F_*W)\le r_{p-1}\frac{\mu(\sF_{p-1})
-\mu(\frac{V_{p-1}}{V_p})}{p\cdot{\rm
rk}(\sE)}-\frac{\mu(\Omega^1_X)}{p\cdot{\rm rk}(\sE)}\cdot(p-1).$$
If $r_{p-1}=0$, there is nothing to prove. If $r_{p-1}>0$, we will
prove $$r_{p-1}\cdot(\mu(\sF_{p-1}) -\mu(V_{p-1}/V_p))\le
\mu(\Omega^1_X)),$$ by using of the following two exact sequences
$$0\to {\rm Sym}^{p-2}(\Omega^1_X)\otimes \omega_X\otimes\sL\to
V_{p-1}/V_p\to {\rm Sym}^p(\Omega^1_X)\otimes\sL\to 0$$
$$0\to \sL\otimes F^*\Omega^1_X\to{\rm Sym}^p(\Omega^1_X)\otimes\sL\to
{\rm Sym}^{p-2}(\Omega^1_X)\otimes \omega_X\otimes\sL\to 0$$ where
all of the bundles have the same slope
$p\cdot\mu(\Omega^1_X)+c_1(\sL)\cdot{\rm H}$.

For $\sF_{p-1}\subset V_{p-1}/V_p$\,, the first exact sequence above
induces an exact sequence $0\to \sF'_{p-1}\to\sF_{p-1}\to
\sF''_{p-1}\to 0$, where
$$\sF'_{p-1}\subset {\rm Sym}^{p-2}(\Omega^1_X)\otimes
\omega_X\otimes\sL, \quad \sF''_{p-1}\subset{\rm
Sym}^p(\Omega^1_X)\otimes\sL.$$ If $\sF''_{p-1}$ is trivial, then we
are done since ${\rm Sym}^{p-2}(\Omega^1_X)\otimes
\omega_X\otimes\sL$ is semi-stable with slope $\mu(V_{p-1}/V_p)$. If
$\sF''_{p-1}\neq 0$, we claim $$r_{p-1}\cdot(\mu(\sF_{p-1})
-\mu(V_{p-1}/V_p))\le {\rm rk}(\sF''_{p-1})\cdot(\mu(\sF''_{p-1})
-\mu(V_{p-1}/V_p)).$$ Indeed, if $\sF'_{p-1}=0$, it is clear. If
$\sF'_{p-1}\neq 0$, we have
$$\mu(\sF_{p-1})=\frac{{\rm
rk}(\sF'_{p-1})}{r_{p-1}}\mu(\sF'_{p-1})+\frac{{\rm
rk}(\sF''_{p-1})}{r_{p-1}}\mu(\sF''_{p-1})$$ and $\mu(\sF'_{p-1})\le
\mu({\rm Sym}^{p-2}(\Omega^1_X)\otimes
\omega_X\otimes\sL)=\mu(V_{p-1}/V_p)$. Put all together, we have the
claimed inequality. Thus it is enough to show $${\rm
rk}(\sF''_{p-1})\cdot(\mu(\sF''_{p-1})
-\mu(V_{p-1}/V_p))\le\mu(\Omega^1_X).$$ The second exact sequence
induces an exact sequence $$0\to E_1\to\sF_{p-1}''\to E_2\to 0$$
where $E_1\subset \sL\otimes F^*\Omega^1_X$, $E_2\subset{\rm
Sym}^{p-2}(\Omega^1_X)\otimes \omega_X\otimes\sL$. If $E_1=0$, it is
clearly done since ${\rm Sym}^{p-2}(\Omega^1_X)\otimes
\omega_X\otimes\sL$ is semi-stable of slope $\mu(V_{p-1}/V_p)$. If
$E_1\neq 0$, by the same argument, we have
$${\rm rk}(\sF''_{p-1})\cdot(\mu(\sF''_{p-1}) -\mu(V_{p-1}/V_p))\le
{\rm rk}(E_1)(\mu(E_1)-\mu(\sL\otimes F^*\Omega^1_X)).$$ If ${\rm
rk}(E_1)=2$, then $E_1=\sL\otimes F^*\Omega^1_X$ and we clearly have
$${\rm rk}(E_1)(\mu(E_1)-\mu(\sL\otimes
F^*\Omega^1_X))=0<\mu(\Omega^1_X).$$ If ${\rm rk}(E_1)=1$, then
$\mu(E_1)-\mu(\sL\otimes F^*\Omega^1_X)\le
\mu_{\max}(\Omega^1_X)=\mu(\Omega^1_X)$ is a special case of
Proposition \ref{prop3.12}, and it is a strict inequality if
$\Omega^1_X$ is stable. To sum up, what we have proved for
$W=\sL\otimes\Omega^1_X$ is
$$\mu(\sE)-\mu(F_*W)\le\left\{
\begin{array}{llll} 0 &\mbox{when $m=2(p-1)$}\\-\frac{\mu(\Omega^1_X)}{p\cdot{\rm rk}(\sE)}\cdot(p-2)
&\mbox{when $m<2(p-1)$}
\end{array}\right.$$
which is a strict inequality if $\Omega^1_X$ is stable.
\end{proof}

\bibliographystyle{plain}

\renewcommand\refname{References}

\end{document}